\documentclass{article}

\usepackage{amsmath}
\usepackage{amsfonts}
\usepackage{amssymb}
\usepackage{color}  
\usepackage{amsthm}
\usepackage{graphicx}
\usepackage{caption}
\usepackage{comment}
\usepackage{subcaption}
\usepackage{algpseudocode}
\usepackage{algorithm}
\usepackage[utf8]{inputenc}
\usepackage{tabularx}
\usepackage{multirow}
\usepackage{hyperref}

\usepackage{natbib}
\bibpunct[, ]{(}{)}{,}{a}{}{,}%

\newtheorem{proposition}{Proposition}
\newtheorem{remark}{Remark}

\title{Smooth and flexible dual optimal inequalities}
\author{Naveed Haghani\textsuperscript{\rm 1,\rm 3}, Claudio Contardo\textsuperscript{\rm 2},  Julian Yarkony\textsuperscript{\rm 3}\\[2ex] 
\textsuperscript{\rm 1}University of Maryland, College Park, MD\\ 
\textsuperscript{\rm 2}ESG UQAM and GERAD, Montreal, Canada\\ 
\textsuperscript{\rm 3}Verisk Computational and Human Intelligence Laboratory, Jersey City, NJ\\ 
}\date{December 2019}

\begin{document}

\maketitle
\begin{abstract}
 We address the problem of accelerating column generation (CG) for set-covering formulations via dual optimal inequalities (DOI). DOI use knowledge of the dual solution space to derive inequalities that might be violated by intermediate solutions to a restricted master problem, and as such are efficient at reducing the number of iterations and the oscillations of the dual variables commonly observed in column generation procedures. We study two novel classes of DOI which are referred to as Flexible DOI (F-DOI) and Smooth-DOI (S-DOI), respectively (and jointly as SF-DOI). F-DOI provide rebates for covering items more than necessary. S-DOI describe the payment of a penalty to permit the under-coverage of items in exchange for the over-inclusion of other items. Unlike other classes of DOI from the literature, the S-DOI and F-DOI rely on very little to no problem-specific knowledge, and as such have the potential to be applied to a vast number of problem domains. In particular, we illustrate the efficiency of the new inequalities by embedding them within a column generation solver for the single source capacitated facility location problem (SSCFLP). A speed-up of a factor of up to 130$\times$ can be observed as when compared to a non-stabilized variant of the same CG procedure to achieve the linear relaxation lower bound on problems with dense columns and structured assignments costs.
\end{abstract}
\section{Introduction}
Column generation (CG) \citep{cuttingstock,lubbecke2005selected,mwspJournal} is a core technique to solve linear programs (LP) with a large number of variables. CG mimics the simplex method by adding a pricing step (just like in the simplex method) to work with a reduced basis and enlarge it only as needed. CG has proven particularly efficient when the underlying mathematical formulation is the result of applying Dantzig-Wolfe decomposition to a compact formulation, in such a way that some complicating constraints are moved to a subproblem with some rich structure. Indeed, it is known that when the pricing subproblem is an integer program without the integrality property \citep{geoffrion1974lagrangean}, the resulting dual bounds may be strictly better than those associated with the original compact formulation of the problem. That is the case of some remarkable applications in logistics and transportation like vehicle routing \citep{costa2019} and facility location \citep{diaz2002branch}.  


CG has been applied historically with great success in a vast number of applications:  stock (material) cutting \citep{cuttingstock}, vehicle routing \citep{costa2019}, crew scheduling \citep{barnprice}, and facility location \citep{barahona1998plant}, among others.  Recently, CG has been applied to computer vision problems including image segmentation \citep{HPlanarCC}, multi-object tracking \citep{wang2017tracking,branchimportant}, and multi-person pose estimation \citep{wang2018accelerating}. When embedded within a branch-and-bound framework, CG becomes branch-and-price, an exact method that guarantees the finding of integer solutions.

Dual optimal inequalities \citep[DOI,][]{mwspJournal, ben2006dual, Gschwind2016Dual} are a core tool to accelerate CG. DOI reduce the dual feasible solution space and thus correspond to relaxations of the constraints in the primal. The dual space is reduced such that it contains at least one dual optimal solution of the non-reduced space. DOI are constructed to be provably inactive at termination of CG. In the primal problem, DOI permit the violation of constraints often in exchange for a penalty. DOI are associated with a repair procedure that maps any solution violating the original primal constraints to one that does not. The solution produced by the repair procedure has no greater cost than that of the original solution.

Typically, DOI exploit the structure of the underlying problems to derive conditions (cuts) satisfied by the dual variables in an optimal solution. For instance, in bin packing and cutting stock items of equal weight/size are indistinguishable. Thus, DOI exploit this knowledge to derive inequalities that impose this equivalence between indistinguishable items in the dual space \citep{ben2006dual, Gschwind2016Dual}. It has not been until recently that \citet{FlexDOIArticle} introduced a new class of DOI for a general class of weighted set-packing problems arising in computer vision called Flexible dual-optimal inequalities (F-DOI). In the primal problem F-DOI permit the over-covering of items in exchange for a penalty. In the dual problem the F-DOI bind the dual variables to lie in a hypercube. The penalty for over-covering an item $v$ with a column $l$ is an upper bound on the increase in cost induced by removing $v$ from any column consisting of a subset of the items of $l$ including $v$.
\section{Our contributions}
Our contributions expand upon previous work on DOI (particularly F-DOI). Most notably, this article contributes to narrowing the gap of applicability of general-purpose DOI by proposing two novel classes of DOI for general set-covering problems. Our contributions can be summarized as follows:
\begin{enumerate}
	\item We introduce a novel class of dual optimal inequality, namely the Smooth-DOI (S-DOI). In the primal problem, S-DOI permit the under-covering of items in exchange for the over-covering of other items plus a penalty. Each individual S-DOI is associated with a pair of items $(u, v)$.  The penalty associated with a given S-DOI for items $(u, v)$ is an upper bound on the increase in cost induced by replacing $u$ in a column not including $v$ with $v$. In the dual problem the S-DOI produce a smoothing in the oscillation behavior of the dual vector, thus the motivation for calling them \textit{Smooth DOI}.
	\item We adapt the F-DOI of \citet{FlexDOIArticle} to weighted set-covering formulations. In our case the F-DOI induce a reward for over-including an element $v$ using column $l$. This reward is a lower bound on the improvement of the objective induced by removing $v$ from any column constructed from a subset of the items of $l$ including $v$.
	\item We provide a unified framework for the simultaneous addition of S-DOI and F-DOI. The resulting inequalities are referred to as \textit{Smooth-Flexible DOI} and denoted SF-DOI.
	\item We propose an efficient implementation for the SF-DOI (and therefore for the S-DOI and F-DOI separately as well), which allow us to control the growth of the resulting linear programs, and to perform pricing in a minimally invasive manner. In fact, we show that the same pricing problem used for a non-stabilized CG can also be applied within the new framework.
	\item We illustrate the efficiency of our stabilization schemes with an application to the single-source capacitated facility location problem (SSCFLP), and show using an exhaustive computational campaign that the new classes of DOI entail substantial improvements in terms of number of iterations and of CPU time required to achieve the linear relaxation lower bounds, as compared to a non-stabilized implementation of the same algorithm.
\end{enumerate}
We outline this document as follows. In Section \ref{section:litrev} we review the most relevant literature in dual stabilization applied to decomposition methods, with an emphasis in CG and DOI. In Section \ref{section:stdcg} we present a standard CG method using a set-covering formulation. In Sections \ref{section:sdoi}-\ref{section:sfdoi} we present the S-DOI, F-DOI and the SF-DOI. In Section \ref{section:effpricing} we consider the problem of pricing in the presence of the new DOI, and show that it can be performed as usual by ignoring the impact of the new DOI in the pricing without compromising the correctness of the method. In Section \ref{section:sscflp} we present a case study by applying the SF-DOI to SSCFLP. In Section \ref{section:results} provide empirical results on SSCFLP. We conclude in Section \ref{section:conclusions}. 

\section{Literature review\label{section:litrev}}

In this section we review the related work on dual stabilization so as to provide context for our contributions. Dual stabilization approaches can be understood as imposing a Bayesian prior on the values of the dual variables and can be classified as trust region based methods, smoothing methods, interior-point methods and dual-optimal inequality based methods. Trust region based methods \citep{marsten1975boxstep,du1999stabilized} prevent or discourage the dual solution from departing a window around the current optimal dual solution. They often require an explicit change in the underlying optimization models and rely on periodic updates of the stabilization parameters, without which the methods are not guaranteed to converge to a proven dual optimum. Interior point methods \citep{Mitchell1994Interior, rousseau2007interior} average across multiple dual optimal solutions to produce well centered solutions. Smoothing methods \citep{Pessoa2018Automation}, on the other hand, do not impose any explicit extra conditions in the models, but rely on the computation of a smoothing dual center to decrease the $\ell_2$-norm of the dual solution used for pricing purposes. As in trust-region methods, a periodic update of the stabilization parameters is often required to ensure convergence. 
DOI based methods \citep{ben2006dual,FlexDOIArticle,Gschwind2017Stabilized}, in contrast, restrict the dual solution from leaving a subspace that provably contains at least one dual optimal solution.
In Sections \ref{section:lr:trust}-\ref{section:lr:smoothing} we provide a brief overview of the first three stabilization techniques. In Section \ref{section:lr:doi} we provide a detailed review of dual-optimal inequalities for CG and related literature. Finally, we  present in Section \ref{section:lr:benders} the application of dual stabilization to a method beyond CG, specifically the Benders decomposition \citep{benders1962partitioning}.

\subsection{Trust region based methods\label{section:lr:trust}}

One key example of trust region based methods is \citet{marsten1975boxstep}, which restricts the dual variables to lie in a box around the previous dual solution at each step of CG. This prevents the dual solution from rapidly traveling outside of the space where the current set of columns provide meaningful information concerning bounds on the dual variables. \citet{marsten1975boxstep} is further developed in \citet{du1999stabilized} where an $\ell_1$ penalty on the distance from a box around the tightest dual solution found thus far is imposed. Here the tightness of the bound is computed by adding negative reduced cost terms to the restricted master problem (RMP) objective so as to provide a lower bound on the optimal solution to the master problem (MP). A key problem with these approaches is the requirement of user defined hyper-parameters for optimization. Related to trust-region methods are Bundle methods \citep{Frangioni2002Generalized, Briant2008Comparison, Ackooij2018Incremental} that allow dual variables to leave the trust region by paying a penalty. Bundle methods make use of non-linear penalty functions (often quadratic) for the penalties, and have desirable convergence properties at the expense of larger computing times per iteration associated with the non-linear penalties.

\subsection{Interior point methods\label{section:lr:interior}}

Interior point methods are based on the assumption that, when multiple dual optima exist for a given problem, the average of those dual vectors tends to be more balanced and thus have lower $\ell_2$ norm. \citet{rousseau2007interior} exploit this observation by solving each iteration of the restricted master problem (RMP) multiple times, to hopefully find alternative dual optima. The average of those optima is used to execute the pricing. \citet{Gondzio2013new} introduce the primal-dual column generation method, which uses an interior point method to avoid solving the RMP to proven optimality at every iteration. Instead, the primal-dual method computes a well-centered primal-dual feasible set of vectors. The associated dual vector is then used as input to the pricing subproblem. Later, \citet{Gondzio2015new} proposed a warm-start strategy to accelerate the re-execution of the primal-dual method. The authors provide further evidence of the usefulness of their method later in \citet{Gondzio2016Large}.

\subsection{Smoothing methods\label{section:lr:smoothing}}

Smoothing in CG can be seen as the use of a dual solution of low $\ell_2$ norm  as input to the pricing sub-problems. Given a dual variable vector $\alpha$ of a linear program, the center $\alpha^0$ is a vector of the same dimensions as $\alpha$ but with a lower $\ell_2$ norm. Instead of using $\alpha$ directly to search for columns of negative reduced cost, one uses instead $\alpha^{\prime} = \lambda\alpha + (1 - \lambda)\alpha^0$, where $\lambda$ is a parameter used to control the smoothing of the vector $\alpha^{\prime}$. \citet{Pessoa2018Automation} presented the use of this technique and provide conditions guaranteeing a correct ---yet automatic--- update of the stabilization parameters to ensure convergence. Unlike trust-region methods, no changes in the structure of the optimization models is necessary along the execution of the algorithm, and therefore is much easier to implement than trust region methods, yet providing competitive and usually better convergence properties.

\subsection{Dual optimal inequalities\label{section:lr:doi}}

The work of \citet{ben2006dual} applies dual optimal inequalities (DOI) for CG. The authors provide conditions satisfied by the dual variables when they are associated to indistinguishable items in the primal, this is items that if swapped, result in a solution totally equivalent. In cutting stock and bin packing, this equivalence is measured according to the items' sizes or weights. Two items of identical weight/size are indistinguishable and the dual variables associated with their coverage restrictions can be assumed to be identical. This entails a series of identity equations in the dual space.  As a result, dual oscillation between indistinguishable items is totally eliminated and the resulting CG algorithms converge substantially faster in the presence of a large amount of such symmetries. The authors also show that in cutting stock/bin packing that for any pair of items $u,v$ ,where the size of $u$ is larger than that of $v$, that the dual variable corresponding to $u$ may be no smaller than that corresponding to $v$. The work of \citep{Gschwind2016Dual} extends this knowledge to problems with richer structure and less symmetries, such as bin packing with conflicts and vector packing.
%

Recently, DOI have made a notable contribution in computer vision applications of CG \citep{mwspJournal}. In the applications considered  by the authors, the fact the cost of a column is a sparse quadratic function of its members is exploited as its cost changes smoothly as observations are removed from it. This observation motivates the following class of DOI, namely the \textit{invariant DOIs}, which is briefly explained as follows. The dual value associated with a given item can be bounded by an easily computed upper bound on the cost increase observed when removing that item from the column. The authors expand on this with \textit{varying DOI} that make use of the current RMP solution to provide bounds that are tighter given the columns in the RMP. Invariant/Varying DOIs can be understood as relaxing the packing constraint and replacing it with a penalty for violating it.  

In the same line as the work of \citet{mwspJournal}, the work of \citet{FlexDOIArticle} introduced the \textit{flexible} DOI (F-DOI) for set packing, which further improves the varying DOIs and is the basis for our work. Varying DOI suffer from the fact that the penalty for using the associated primal variable depends only on the item being removed, and as such do not exploit any specific knowledge of the active columns in the RMP solution. F-DOI circumvent this issue by considering that the penalty may depend not only on the item, but also on the column, to compute bounds on the cost change when removing an item from that specific column.


DOI can also be found in (hierarchical) correlation clustering problems arising in image segmentation \citep{HPlanarCC,PlanarCC}. The authors report that in the absence of dual stabilization, CG suffers from severe convergence issues. In the presence of DOI, on the other hand, convergence is achieved several orders of magnitude faster. In \citet{PlanarCC} DOIs are able to bound most of the dual variables to be zero for real problem instances. This is not the case in \citet{HPlanarCC}, which is closely related to \citet{PlanarCC}, implying that it is not only the restriction of the dimensionality of the dual space which provides for the effectiveness of DOIs in \citet{PlanarCC}. 
%
%
\subsection{Dual stabilization in Benders decomposition\label{section:lr:benders}}

Dual stabilization is not only a relevant issue in CG. It also occurs in the domain of Benders decomposition \citep{benders1962partitioning} as a part of Magnanti-Wong cuts \citep{magnanti1981accelerating}.   Magnanti-Wong cuts are Benders cuts that are optimal (or near optimal) with respect to the original objective and yet attempt to optimize a separate auxiliary objective.  These cuts have been adapted to computer vision problems for multi-person pose estimation \citep{wang2018accelerating} and image segmentation \citep{keuper2019massively}. In logistics applications, dual stabilization of Benders decomposition has proven effective in hub-location \citep{Contreras2011Benders} and network design \citep{Zetina2019Exact}. In these settings the auxiliary objective may be an $\ell_1$ norm as in \citet{wang2018accelerating}, or a random objective \citep{keuper2019massively}.  Remarkably these Magnanti-Wong cuts are often necessarily for convergence of Benders decomposition as reported in \citet{keuper2019massively}.  
\section{Standard column generation\label{section:stdcg}}
In this section we present a standard CG method over a set-covering formulation. We are given a set $N$ of items to be covered. The formulation includes a continuous variable $\theta_l\geq 0$ for every column $l\in\Omega$ where $\Omega$ is the set of all feasible columns. It is assumed that every column can cover an item at most once. In the context of CG, the set $\Omega$ typically is exponentially large in the size of the input, and therefore it is impractical to explicitly consider the entire set during optimization. For every column $l\in\Omega$ and for every item $u\in N$ we let $a_{ul}\in\{0, 1\}$ be a binary constant equal to 1 if $l$ covers $u$, and otherwise let $a_{ul}=0$. We consider problem:
\begin{equation}\label{setcover:cost}
    \min_{\theta}\qquad \qquad \sum_{l\in\Omega} c_l\theta_l
\end{equation}
subject to
\begin{align}
    & \sum_{l\in\Omega}a_{ul}\theta_l \geq 1 & u\in N\label{setcover:cover}\\
    & \theta_l\geq 0 & l\in\Omega.\label{setcover:nature}
\end{align}

Given the unscalability of enumerating the set $\Omega$ explicitly, CG consists in considering a subset $\Omega_R\subset\Omega$ at any given time, and thus the RMP consists in the solution of problem \eqref{setcover:cost}-\eqref{setcover:nature} but restricted to the columns in $\Omega_R$. Let us denote that problem $\mathtt{RMP}(\Omega_R)$. Now let us denote $(\alpha_u)_{u\in N}$ as the dual variables associated with the constraints \eqref{setcover:cover} in problem $\mathtt{RMP}(\Omega_R)$. The reduced cost of a column $l\in\Omega\setminus\Omega_R$ can be computed as $\overline{c_l} = c_l - \sum_{u\in N}\alpha_u a_{ul}$. Problem $\mathtt{RMP}(\Omega_R)$ provides a proven optimal solution of \eqref{setcover:cost}-\eqref{setcover:nature} if $\min\{\overline{c_l} : l\in\Omega\setminus\Omega_R\} \geq 0$. 

\section{Smooth DOI\label{section:sdoi}}

Let $s = (u, v)\in N\times N, u\neq v$ be a pair of items satisfying the following property: for every column $l\in\Omega$ such that $a_{ul} = 1, a_{vl} = 0$ the column $l'$ constructed from swapping $u$ for $v$ from $l$ also lies in $\Omega$ (the converse may not be true). Let us denote $l' = swap(l, s)$ as the outcome of the above swap operation, and let us denote by $\Omega(s)$ the set of columns $l\in\Omega$ such that $a_{ul} = 1, a_{vl} = 0$. Let us define, for a given $s$ satisfying this property, a penalty $\rho_s\in\mathbb{R}$ such that $\rho_s\geq \max\{c_{l'} - c_l: l\in\Omega(s), l'= swap(l, s)\}$. Let us consider a set $\mathcal{S}$ such that for every $s_1 = (u, v), s_2 = (v, w)\in S$, that $s_3 = (u, w)$ is also in $\mathcal{S}$ and $\rho_{s_1} + \rho_{s_2} \geq \rho_{s_3}$ holds (triangle inequality). For a pair $s = (u, v)$, we denote $s^+ = v$ and $s^- = u$. Moreover, we denote $\mathcal{S}^-_u = \{s\in\mathcal{S}: s^- = u\}$, $\mathcal{S}^+_v = \{s\in\mathcal{S}: s^+ = v\}$.

Let us consider the following stabilized form of problem \eqref{setcover:cost}-\eqref{setcover:nature}. In addition to the notation introduced before, we add one non-negative variable $\omega_s$ for every $s\in\mathcal{S}$. We consider the problem as follows:
\begin{equation}\label{smooth:cost}
    \min_{\theta, \omega}\qquad \qquad \sum_{l\in\Omega} c_l\theta_l + \sum_{s\in\mathcal{S}}\rho_s\omega_s
\end{equation}
subject to
\begin{align}
    & \sum_{l\in\Omega}a_{ul}\theta_l + \sum_{s\in \mathcal{S}^+_u} \omega_s - \sum_{s\in \mathcal{S}^-_u} \omega_s \geq 1 & u\in N\label{smooth:cover}\\
    & \theta_l\geq 0 & l\in\Omega.\label{smooth:nature1}\\
    & \omega_s\geq 0 & s\in\mathcal{S}.\label{smooth:nature2}
\end{align}

\begin{proposition}
Problem \eqref{smooth:cost}-\eqref{smooth:nature2} admits an optimal solution $(\theta^*, \omega^*)$ such that $\omega^*_s = 0$ for every $s\in\mathcal{S}$. 
\end{proposition}
\begin{proof}
Let $(\theta^*, \omega^*)$ be an optimal solution to problem \eqref{smooth:cost}-\eqref{smooth:nature2}. If multiple optima exist, let $(\theta^*, \omega^*)$ be the one attaining the lowest possible value of $\Sigma(\theta^*, \omega^*) = \sum\{\theta^*_l:l\in\Omega\} + \sum\{\omega^*_s: s\in\mathcal{S}\}$. We will prove, by contradiction, that no variable $\omega^*_s$ can take a strictly positive value. Let us assume that $\omega^*_s > 0$ for a certain $s\in\mathcal{S}$ and let us consider the following three scenarios:
\begin{itemize}
    \item \textbf{There exists $t\in\mathcal{S}$ such that $\omega^*_t > 0, s^- = t^+, s^+ \neq t^-$}. 
    
    Let $r = (t^-, s^+)$. Because of the triangle inequality, $r$ lies in $S$ and $\rho_r\leq \rho_s + \rho_t$. Let $\Delta = \min\{\omega^*_s, \omega^*_t\}$. We let $\omega^{\prime}_s\leftarrow \omega^*_s - \Delta$, $\omega^{\prime}_t\leftarrow \omega^*_t - \Delta$, $\omega^{\prime}_r\leftarrow \omega^*_r + \Delta$ and $\omega^{\prime}_{k} \leftarrow \omega^*_{k}$ for all $k\in\mathcal{S}\setminus\{s, t, r\}$. It is easy to see that $(\theta^*, \omega^{\prime})$ is primal feasible. The marginal contribution of replacing $\omega^*$ by $\omega^{\prime}$ is $\Delta(\rho_r - \rho_s - \rho_t)$, which is non-positive. If negative (meaning that $\Delta(\rho_r - \rho_s - \rho_t) < 0$) we obtain a contradiction with the optimality of $(\theta^*, \omega^*)$. If zero, on the other hand, the operation provides an alternate optimal solution $(\theta^*, \omega^{\prime})$ such that $\Sigma(\theta^*, \omega^{\prime}) < \Sigma(\theta^*, \omega^*)$, which is also a contradiction. 
    
    \item \textbf{There exists $t\in\mathcal{S}$ such that $\omega^*_t > 0, s^- = t^+, s^+ = t^-$}.
    
    In this case we let $\Delta = \min\{\omega^*_s, \omega^*_t\}$ and let $\omega^{\prime}_s\leftarrow \omega^*_s - \Delta$, $\omega^{\prime}_t\leftarrow \omega^*_t - \Delta$, and $\omega^{\prime}_{k} \leftarrow \omega^*_{k}$  for all $k\in\mathcal{S}\setminus\{s, t\}$. It is easy to see that the solution $(\theta^*, \omega^{\prime})$ is also primal feasible and the marginal contribution of this operation is $-\Delta(\rho_s + \rho_t)$. The term $(\rho_s + \rho_t)$ is non-negative because of the following observation. From the definition of $\Omega(\cdot)$ we have that $l\in\Omega(s), l^{\prime} = swap(l, s) \Leftrightarrow l^{\prime}\in\Omega(t), l = swap(l', t)$. Then, for any such pair $(l, l')$ we have ---from the conditions satisfied by $\rho$--- that $\rho_s + \rho_t\geq [c_{l^{\prime}} - c_l] + [c_l - c_{l^{\prime}}]\geq 0$. If $\rho_s + \rho_t > 0$, this operation results in a contradiction with the optimality of $(\theta^*, \omega^*)$. If zero, on the other hand, the solution obtained is an alternate optimum and $(\theta^*, \omega^{\prime})$ is such that $\Sigma(\theta^*, \omega^{\prime}) < \Sigma(\theta^*, \omega^*)$, which is also a contradiction.
    
    \item \textbf{For every $t\in\mathcal{S}$ such that $s^- = t^+$, $\omega_t = 0$}.
    
    Since the primal problem is feasible and no $t\in\mathcal{S}$ exists satisfying $\omega^*_t > 0, s^- = t^+$, there must exist a column $l\in\Omega(s)$ such that $\theta^*_l > 0$. Let $\Delta = \min\{\omega^*_s, \theta^*_l\}$ and let $l^{\prime} = swap(l, s)$. We construct new variables $\theta^{\prime}, \omega^{\prime}$ with all components equal to those of $(\theta^*, \omega^*)$ except for $\omega^{\prime}_s\leftarrow \omega^*_s - \Delta$, $\theta^{\prime}_l\leftarrow \theta^*_l - \Delta$, $\theta^{\prime}_{l^{\prime}}\leftarrow \theta^*_{l^{\prime}} + \Delta$. This operation entails an increase in the objective of $\Delta(c_{l^{\prime}} - c_l - \rho_s)$ which by definition of $\rho_s$ is non-positive. If negative, this would contradict the optimality of $(\theta^*, \omega^*)$. If zero, on the other hand, it would entail an alternate optimum such that $\Sigma(\theta^{\prime}, \omega^{\prime}) < \Sigma(\theta^*, \omega^*)$ which is also a contradiction.
\end{itemize}
\end{proof}

\section{Flexible DOI\label{section:fdoi}}

In this section we present the Flexible DOI (F-DOI), which use the potential benefit attained when removing items from a column to bound the solution space of the dual variables $(\alpha_u)_{u\in N}$. In Section \ref{section:fdoi:general} we introduce F-DOI in the general case. In Section \ref{section:fdoi:accel} we present an efficient implementation of the F-DOI to address some scalability issues.

\subsection{The general case\label{section:fdoi:general}}

For a column $l\in \Omega$ and an item $u\in N$ we let $\sigma_{ul}\geq 0$ such that $a_{ul} = 0\Rightarrow \sigma_{ul} = 0$. The set of values $(\sigma_{ul})_{u\in N, l\in\Omega}$ is assumed to satisfy the following property. For every column $l\in \Omega_R$ and for every subset $S\subseteq N(l) = \{u\in N: a_{ul} = 1\}$, $\sum_{u\in S}\sigma_{ul}\leq c_l - c_{l'}$, where $l'\in\Omega$ is the column constructed from removing the items in $S$ from $l$. We denote $l' = remove(l, S)$. For any given $u\in N$, we let $\Lambda_u = \{\sigma_{ul}: l\in\Omega_R\}$ be the set of all possible values of $\sigma_{ul}$ across the restricted set $\Omega_R$.

In addition to the notation introduced in Section \ref{section:stdcg}, let us introduce non-negative variables $\xi_{u\sigma}$ for every $u\in N$ and for every $\sigma\in\Lambda_u$. A variable $\xi_{u\sigma}$ represents the number of columns $l$ from which $u$ will be removed with a rebate of $\sigma$. We also let $\beta_{ul\sigma}$ be a binary constant equal to 1 if $\sigma_{ul} = \sigma$. The following is a stabilized version of \eqref{setcover:cost}-\eqref{setcover:nature}:
\begin{equation}\label{flexible:cost}
    \min_{\theta, \xi}\qquad \qquad \sum_{l\in\Omega} c_l\theta_l - \sum_{u\in N, \sigma\in\Lambda_u}\sigma\xi_{u\sigma}
\end{equation}
subject to
\begin{align}
    & \sum_{l\in\Omega}a_{ul}\theta_l - \sum_{\sigma\in\Lambda_u} \xi_{u\sigma} \geq 1 & u\in N\label{flexible:cover}\\
    & \xi_{u\sigma} - \sum_{l\in\Omega}\beta_{ul\sigma}\theta_l \leq 0 & u\in N, \sigma\in\Lambda_u\label{flexible:bound}\\
    & \theta_l\geq 0 & l\in\Omega\label{flexible:nature1}\\
    & \xi_{u\sigma}\geq 0 & u\in N, \sigma\in\Lambda_u.\label{flexible:nature2}
\end{align}

Intuitively, the F-DOI allow the over-coverage of some items for a reward in the objective. If this reward is sufficiently small, it will not be sufficient to compensate the over-coverage and thus will be non-basic in an optimal solution of problem \eqref{flexible:cost}-\eqref{flexible:nature2}. The following proposition formalizes this result.
\begin{proposition}\label{prop:fdoi}
	Problem \eqref{flexible:cost}-\eqref{flexible:nature2} admits an optimal solution $(\theta^*, \xi^*)$ such that $\xi^*_{u\sigma} = 0$ for every $u\in N, \sigma\in\Lambda_u$.
\end{proposition}
\begin{proof}
	Let $(\theta^*, \xi^*)$ be an optimal solution to problem \eqref{flexible:cost}-\eqref{flexible:nature2}. If multiple such optima exist, let $(\theta^*, \xi^*)$ be the one that minimizes $\Sigma(\theta, \xi) = \sum\{\theta_l: l\in\Omega\} + \sum\{\xi_{u\sigma}: u\in N, \sigma\in\Lambda_u\}$. We will prove by contradiction that if $\xi^*_{u\sigma} > 0$ for some $u\in N, \sigma\in\Lambda_u$ then either $(\theta^*, \xi^*)$ cannot be optimum, or that an alternate optimum $(\theta^{\prime}, \xi^{\prime})$ exists such that $\Sigma(\theta^{\prime}, \xi^{\prime}) < \Sigma(\theta^*, \xi^*)$.
	
	Let $\xi^*_{u\sigma} > 0$ for some $u\in N, \sigma\in\Lambda_u$. Constraints \eqref{flexible:bound} assure the existence of a column $l\in\Omega$ such that $\beta_{ul\sigma} = 1$ and $\theta^*_l > 0$. Let $\Delta = \min\{\xi^*_{u\sigma}, \theta^*_l\}$. Let $l'=remove(l, \{u\})$ be the column resulting from removing $u$ from $l$. Now, let us consider a solution $(\theta^{\prime}, \xi^{\prime})$ with all entries equal to those of $(\theta^*, \xi^*)$ except for the entries $\xi^{\prime}_{u\sigma}\leftarrow \xi^*_{u\sigma} - \Delta$, $\theta^{\prime}_{l}\leftarrow \theta^*_{l} - \Delta$, $\theta^{\prime}_{l^{\prime}}\leftarrow \theta^*_{l^{\prime}} + \Delta$. This operation entails a feasible solution with a marginal contribution to the objective equal to $\Delta(c_{l^{\prime}} - c_l + \sigma)$. Either this quantity is negative ---which would contradict the optimality of $(\theta^*, \xi^*)$--- or $\Sigma(\theta^{\prime}, \xi^{\prime}) < \Sigma(\theta^{*}, \xi^{*})$ which is also not possible.
\end{proof}

\subsection{An efficient implementation\label{section:fdoi:accel}}


Problem \eqref{flexible:cost}-\eqref{flexible:nature2} may contain an exponentially large number of $\xi$ variables as indexed by all the possible values of $\sigma$. 
Let $\Lambda^R_u\subseteq \Lambda_u$ be a (typically much) smaller subset of $\sigma$ values which includes the largest value in $\Lambda_u$. Because any rebate vector $\sigma^{\prime}$ such that $\sigma^{\prime}_{ul}\leq \sigma_{ul}$ for every $u\in N, l\in\Omega$ is a valid rebate vector, one could replace every rebate $\sigma\in\Lambda_u$ by $\sigma^{\prime}\leftarrow\max\{\sigma^{\prime}\in \Lambda^R_u: \sigma^{\prime} \leq \sigma\}$. Therefore, one may consider problem \eqref{flexible:cost}-\eqref{flexible:nature2} by replacing $\Lambda_u$ by $\Lambda^R_u$, thus reducing the number of variables $\xi$ and the number of constraints \eqref{flexible:bound}, \eqref{flexible:nature2} without compromising the conditions required for Proposition \ref{prop:fdoi}. Such a set $\Lambda^R_u$ can be constructed by selecting the $M$ quantities in $\Lambda_u$, where $M$ is a user defined parameter ($M$=20 works well in practice).  Observe that adding columns to the RMP changes the quintiles and hence would seem to require resolving the RMP from scratch during each iteration of CG.  Since resolving the RMP from scratch at each iteration of CG is  too time intensive we only change the quintiles periodically.  We add the value zero to $\Lambda^R_u$ $ \forall u \in \mathcal{N}$ so as to model the case of a value $\sigma$ that is smaller than any quintile value.

\section{Smooth-Flexible DOI\label{section:sfdoi}}

In this section we address the problem of combining both S-DOI and F-DOI within an unified framework. We refer to the resulting inequalities as Smooth-Flexible DOI and denote them as SF-DOI. Using the same notation and decision variables introduced in Sections \ref{section:stdcg}-\ref{section:fdoi}, we consider the following optimization problem:
\begin{equation}\label{sfdoi:cost}
\min_{\theta, \omega, \xi}\qquad \qquad \sum_{l\in\Omega} c_l\theta_l + \sum_{s\in\mathcal{S}}\rho_s\omega_s - \sum_{u\in N, \sigma\in\Lambda_u}\sigma\xi_{u\sigma}
\end{equation}
subject to
\begin{align}
& \sum_{l\in\Omega}a_{ul}\theta_l + \sum_{s\in \mathcal{S}^+_u} \omega_s - \sum_{s\in \mathcal{S}^-_u} \omega_s - \sum_{\sigma\in\Lambda_u} \xi_{u\sigma} \geq 1 & u\in N\label{sfdoi:cover}\\
& \xi_{u\sigma} - \sum_{l\in\Omega}\beta_{ul\sigma}\theta_l \leq 0 & u\in N, \sigma\in\Lambda_u\label{sfdoi:bound}\\
& \theta_l\geq 0 & l\in\Omega\label{sfdoi:nature1}\\
& \omega_s\geq 0 & s\in\mathcal{S}\label{sfdoi:nature2}\\
& \xi_{u\sigma}\geq 0 & u\in N, \sigma\in\Lambda_u.\label{sfdoi:nature3}
\end{align}

The following proposition formalizes the correctness of the SF-DOI by proving that, at optimality, variables $\omega, \xi$ are non-basic.
\begin{proposition}
	Problem \eqref{sfdoi:cost}-\eqref{sfdoi:nature3} admits an optimal solution $(\theta^*, \omega^*, \xi^*)$ such that $\omega^*_s = 0$ for every $s\in\mathcal{S}$ and that $\xi^*_{u\sigma} = 0$ for every $u\in N, \sigma\in\Lambda_u$.
\end{proposition}
\begin{proof}
	Let $(\theta^*, \omega^*, \xi^*)$ be an optimal solution to problem \eqref{sfdoi:cost}-\eqref{sfdoi:nature3}. If multiple optima exist, let it be one that minimizes $\Sigma(\theta, \omega, \xi) = \sum_{l\in\Omega}\theta_l + \sum_{s \in \mathcal{S}}\omega_s + \sum_{u\in N, \sigma\in\Lambda_u}\xi_{u\sigma}$. The proof follows by applying the same arguments provided for the proofs of correctness for the S-DOI and F-DOI in sequence. First we assume that $\omega^*_s > 0$ for some $s\in\mathcal{S}$ to arrive to a contradiction. Then, assuming that $\omega^*_s = 0$ for every $s\in\mathcal{S}$, we assume that $\xi^*_{u\sigma} > 0$ for some $u\in N, \sigma\in\Lambda_u$ to arrive to another contradiction.
\end{proof}

\begin{remark}
	The same arguments used to reduce the problem size associated with the F-DOI can be applied to reduce the problem size of \eqref{sfdoi:cost}-\eqref{sfdoi:nature3} by replacing $\Lambda$ by a restricted subset $\Lambda^R_u\subseteq\Lambda_u$.
\end{remark}

\section{Efficient pricing\label{section:effpricing}}

Let us consider the mathematical problem \eqref{sfdoi:cost}-\eqref{sfdoi:nature3} including the SF-DOI. Let $(\alpha, \gamma)$ be a vector of dual variables associated with the constraints \eqref{sfdoi:cover}-\eqref{sfdoi:bound}. For a column $l\in\Omega$, its reduced cost $\overline{c_l}$ can be computed as follows:
\begin{equation}\label{sfdoi:redcost}
\overline{c_l} = c_l - \sum_{u\in N} a_{ul}\alpha_u + \sum_{u\in N, \sigma\in\Lambda_u}\beta_{ul\sigma}\gamma_{u\sigma}.
\end{equation}

As shown by \citet{FlexDOIArticle} in the context of set-packing, the dual variables $\gamma$ can be ignored by the pricing algorithm, yet without compromising the correctness of the CG method. This result is important since it allows us to use a standard pricing algorithm instead of having to adapt one for the consideration of the new dual variables $\gamma$. The following proposition formalizes this result.
\begin{proposition}
	Let $(\alpha^*, \gamma^*)$ be an optimal solution to the dual RMP obtained from dualizing problem \eqref{sfdoi:cost}-\eqref{sfdoi:nature3} and where $\Omega$ is replaced by a restricted set $\Omega_R$, and such that $\min\{c_l - \sum_{u\in N} a_{ul}\alpha^*_u: l\in\Omega\} \geq 0$. Then $\sum_{u\in N}\alpha^*_u = z^*$, where $z^*$ is the optimal value associated with problem \eqref{setcover:cost}-\eqref{setcover:nature}.
\end{proposition}
\begin{proof}
	Since $\Omega_R\subseteq\Omega$ we have that $z^*\leq \sum_{u\in N}\alpha^*_u$. Also, since $\min\{c_l - \sum_{u\in N} a_{ul}\alpha^*_u: l\in\Omega\} \geq 0$ it follows that $\alpha^*$ is also feasible for the dual of problem \eqref{setcover:cost}-\eqref{setcover:nature}, which in turn implies that $\sum_{u\in N}\alpha^*_u \leq z^*$.
\end{proof}

\section{Case study: single-source capacitated facility location\label{section:sscflp}}

In this section we provide an efficient implementation of S-DOI, F-DOI and SF-DOI to the solution of the single-source capacitated facility location problem (SSCFLP). CG remains one of the most efficient methods to compute high quality dual bounds for the SSCFLP \citep{diaz2002branch}, although recently cut-and-solve methods have proven particularly efficient for its solution \citep{yang2012cut, Gadegaard2018improved}. The motivation to consider the SSCFLP in this case study comes from the observation made by other authors in the past that CG applied for the SSCFLP suffers from serious degeneracy issues for problems with dense optimal columns \citep{Gadegaard2018improved}.

In the SSCFLP, we are given a set of customers $N$ and a set of potential facilities $I$, and such that $I\cap N=\emptyset$. With each facility $i\in I$ is associated a fixed opening cost $f_i$ and a capacity $K_i$. With each customer $u\in N$ is associated a demand $d_u\geq 0$. With each pair $(i, u)\in I\times N$ is associated an assignment cost $c_{iu}\geq 0$. Without loss of generality, we may assume that all parameters are integer-valued. It is assumed that $\sum\{K_i: i\in I\}\geq \sum\{d_u: u\in N\}$ as otherwise the problem is trivially infeasible. The problem consists in selecting a subset $I^{\prime}\subseteq I$ of facilities to open, and to assign every customer to exactly one open facility in such a way that the capacities of the selected facilities are respected, at minimum total cost (fixed costs + assignment costs).

The SSCFLP can be modeled as a set-covering problem, as follows. For every facility $i\in I$, let $\Omega_i$ be the set of pairs $(i, S)$ where $S\subseteq N$ is a customer subset that can be assigned fully to $i$ without it exceeding its capacity (thus such that $\sum\{d_u:u\in S\}\leq K_i$). The set of all such pairs $(i, S)$ is denoted $\Omega$. The cost of an assignment pattern $l = (i, S)$ is given by $c_l = f_i + \sum\{c_{iu}: u\in S\}$. For an assignment pattern $l = (i, S)$ let $a_{ul}\geq 0$ be a binary constant taking the value 1 if $u\in S$. Let $\theta_l\geq 0$ be a continuous variable representing the number of times that the assignment pattern $l$ appears in the solution. The set-covering formulation for the SSCFLP is as follows:
\begin{equation}\label{sscflp:cost}
\min_{\theta}\qquad \sum_{l\in\Omega}c_l\theta_l 
\end{equation}
subject to
\begin{align}
& \sum_{l\in\Omega}a_{ul}\theta_l \geq 1 & u\in N\label{sscflp:covering}\\
& \sum_{l\in\Omega_i}\theta_l \leq 1 & i\in I\label{sscflp:bound}\\
& \theta_l \geq 0 & l\in\Omega.
\end{align}

To embed this problem within the proposed framework, we need to provide efficient mechanisms to compute the rebates $\sigma$ (for the F-DOI) and the penalties $\rho$ (for the S-DOI). We propose the following formulaes for the computation of those quantities:
\begin{align}
\rho_{uv} & \leftarrow \max\{c_{iv} - c_{iu}: i\in I\} & u, v\in N, u\neq v, d_u \geq d_v\\
\sigma_{ul} & \leftarrow c_{iu} & u\in N, l = (i, S)\in\Omega.
\end{align}

It is easy to see that these quantities satisfy the conditions necessary for the validity of the SF-DOI using formulation \eqref{sfdoi:cost}-\eqref{sfdoi:nature3}. Moreover, they are computationally cheap to compute. The pricing subproblem, as usual for the SSCFLP, remains a 0-1 knapsack problem.

\section{Computational results}
\label{section:results}

In this section we present computational evidence of the impact of the DOI introduced in this article for the particular case of the SSCFLP. We consider two classical benchmark datasets ---namely the \citeauthor{holmberg1999exact} and the \citeauthor{yang2012cut} datasets--- plus two additional synthetic datasets generated by us. The performance of the DOI is assessed according to the speedup provided with respect to a classical (non-stabilized) CG algorithm for the same problem. Each algorithm is executed until proving optimality of the CG, thus  the linear relaxation is solved fully. The algorithms have been coded in Matlab and we use CPLEX as general-purpose MIP solver for solving the RMP. Our machine is equipped with a 8-core AMD Ryzen 1700 CPU @3.0 GHz and 32 GB of memory running Windows 10.




When employing F-DOI we assign the values in $\Lambda_u^R$ as 20 evenly spaced quantiles over the distribution of values in $\Lambda_u$. As more columns enter the RMP this distribution changes.  We update $\Lambda_u^R$ periodically to reflect this change and more accurately represent the distribution.  We update on iterations 1, 5, 25, 100, 200, 500 and every 500 iterations onward. When employing S-DOI we save computational time in solving the RMP by only including a subset of the DOI. Specifically, we include only the DOI variables associated to the smallest 25\% of $\rho_s$ values.

For pricing, each facility induces a 0-1 knapsack problem capable of producing a negative reduced cost column. We choose to return the first 20 columns with negative reduced cost found while iterating over the facilities. If all facilities are cycled through before 20 negative reduced cost columns are found, the algorithm terminates pricing and returns all the negative reduced cost columns found by that point. If through a single round pricing no negative reduced cost columns are found after all facilities have been cycled through, then pricing is terminated and CG is then complete. The 0-1 knapsack problems are solved using dynamic programming.

\subsection{Results on the \citeauthor{holmberg1999exact} dataset}

In our first set of experiments we test our DOI on problems from the SSCFLP benchmark dataset defined in \citep{holmberg1999exact}. We focus on the 16 largest problems (numerically indexed 56-71 in the original paper) where $|N|=200$ and $|I|=30$. In those instances, the capacities and demands are assigned randomly such that the ratio of total capacity over total demand ($K_{total}/d_{total}$) ranges from 1.97 to 3.95. The facility fixed costs are distributed over a range from 500 to 1500 and the assignment costs are proportional to the Euclidean distance between a customer and a facility.

We execute each algorithm variant on all 16 problem instances in this dataset. In Table \ref{table:TB2} we report the CPU times required to achieve the linear relaxation lower bound of problem \eqref{sscflp:cost}-\eqref{sscflp:bound} for each algorithm variant, and the associated speedup as a a factor of the CPU taken for the standard CG. In Figure \ref{fig:TB2} we report the following data. In the two top figures, we plot average gap (across all 16 problem instances) as a function of the CPU time (left-most figure) and the number of iterations (right-most figure). The two bottom figures report the speedup (the time ratio of employing a given DOI as a factor of employing no DOI) as a function of run time (left-most figure) and number of iterations (right-most figure). All three DOI on average provide a speedup over standard CG with S-DOI performing the best with an average speedup of 3.3. The S-DOI provides a positive speedup for all 16 instances while the F-DOI and the SF-DOI show more mixed results in a case by case analysis. The F-DOI provide a positive speedup in 11 out of the 16 instances while the SF-DOI provide a positive speedup in 12 out of the 16 instances. All three DOI categorically required fewer CG iterations to converge than the standard CG, however employing these DOI comes at greater computational cost in solving the RMP. This computational liability is enough to negate the benefit of the DOI in those instances where the DOI provided no benefit. The SF-DOI notably required the fewest iterations to converge on average but still, for the most part, is outperformed by the S-DOI when judging for time. We note that employing the F-DOI largely comes at a significantly greater computational cost than employing the S-DOI, which accounts for its lower performance benefit.

\begin{table}[!hbtp]
	\centering
	\scalebox{0.8}{
		\begin{tabular}{|c|c c c c|c c c|} 
		\hline
		& \multicolumn{4}{c|}{\bf Time (sec)} & \multicolumn{3}{c|}{\bf Speedup}\\
		\bf Instance & \bf Standard & \bf S-DOI & \bf F-DOI & \bf SF-DOI & \bf S-DOI & \bf F-DOI & \bf SF-DOI \\
		\hline
		56 & 9.3 & 5.9 & 11.7 & 9.5 & 1.6 & 0.8 & 1.0 \\ 
		57 & 9.4 & 6.9 & 21.8 & 21.8 & 1.4 & 0.4 & 0.4 \\
		58 & 9.8 & 8.6 & 48.5 & 41.3 & 1.1 & 0.2 & 0.2 \\
		59 & 15.4 & 10.2 & 42.6 & 34.6 & 1.5 & 0.4 & 0.4 \\
		60 & 32.8 & 8.8 & 10.9 & 8.5 & 3.7 & 3.0 & 3.8 \\ 
		61 & 43.7 & 9.4 & 21.8 & 10.8 & 4.6 & 2.0 & 4.0 \\
		62 & 33.7 & 11.4 & 45.1 & 43.8 & 3.0 & 0.7 & 0.8 \\
		63 & 69.0 & 12.9 & 21.0 & 11.8 & 5.3 & 3.3 & 5.8 \\
		64 & 55.5 & 10.6 & 14.9 & 6.0 & 5.2 & 3.7 & 9.2 \\ 
		65 & 87.1 & 13.1 & 16.5 & 11.0 & 6.6 & 5.3 & 7.9 \\
		66 & 55.1 & 16.6 & 37.8 & 41.6 & 3.3 & 1.5 & 1.3 \\
		67 & 72.5 & 66.3 & 25.1 & 18.0 & 1.1 & 2.9 & 4.0 \\
		68 & 29.4 & 8.5 & 13.9 & 10.4 & 3.5 & 2.1 & 2.8 \\ 
		69 & 44.9 & 11.2 & 19.2 & 10.3 & 4.0 & 2.3 & 4.4 \\
		70 & 81.2 & 18.5 & 68.3 & 30.3 & 4.4 & 1.2 & 2.7 \\
		71 & 45.1 & 15.0 & 27.8 & 21.3 & 3.0 & 1.6 & 2.1 \\
		\hline
		\bf mean & \bf43.4 & \bf14.6 & \bf27.9 & \bf20.7 & \bf3.3 & \bf2.0 & \bf3.2 \\
		\bf median & \bf44.3 & \bf10.9 & \bf21.8 & \bf14.9 & \bf3.4 & \bf1.8 & \bf 2.8 \\
		\hline
	\end{tabular}}
	\caption{\citeauthor{holmberg1999exact} dataset $(|N|=200)$ runtime results}
	\label{table:TB2}
\end{table}

\begin{figure}[!hbtp]
	\includegraphics[width=0.45\linewidth]{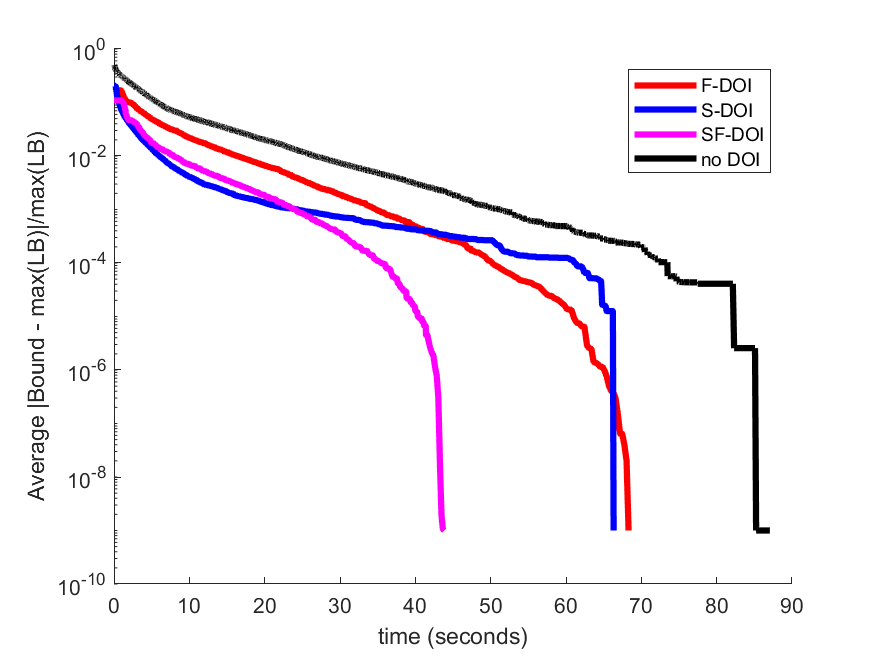}
	\includegraphics[width=0.45\linewidth]{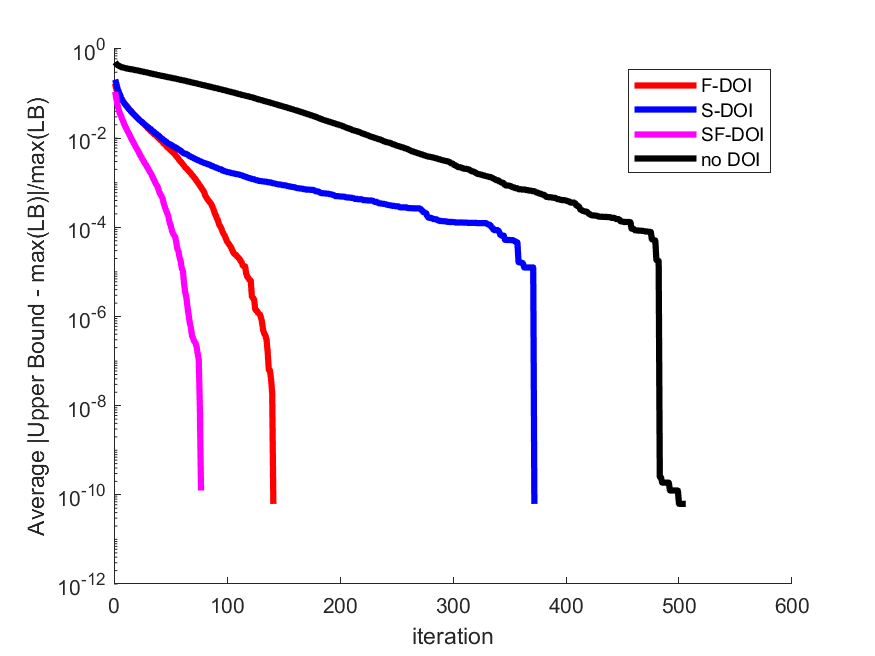}\\
	\includegraphics[width=0.45\linewidth]{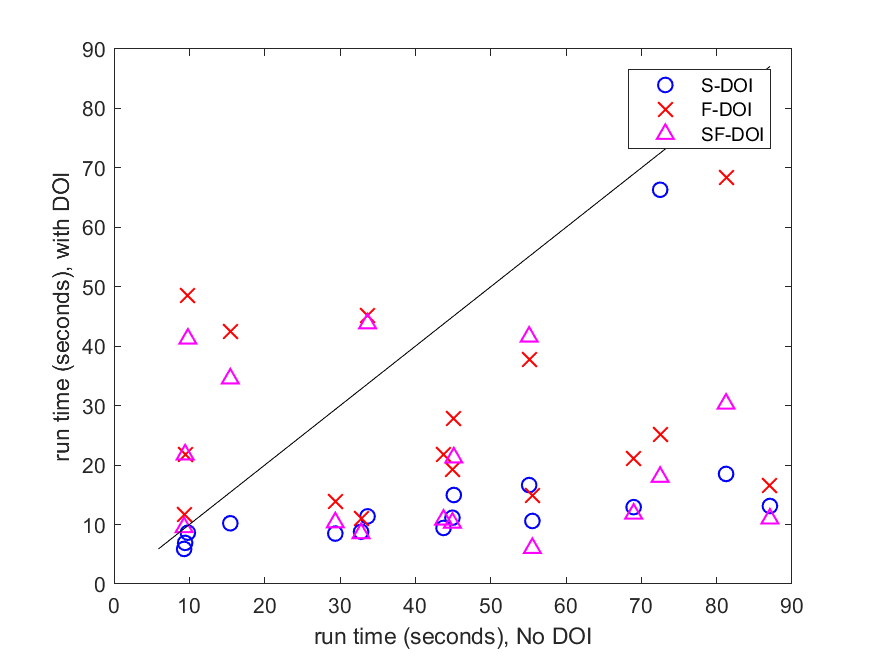}
	\includegraphics[width=0.45\linewidth]{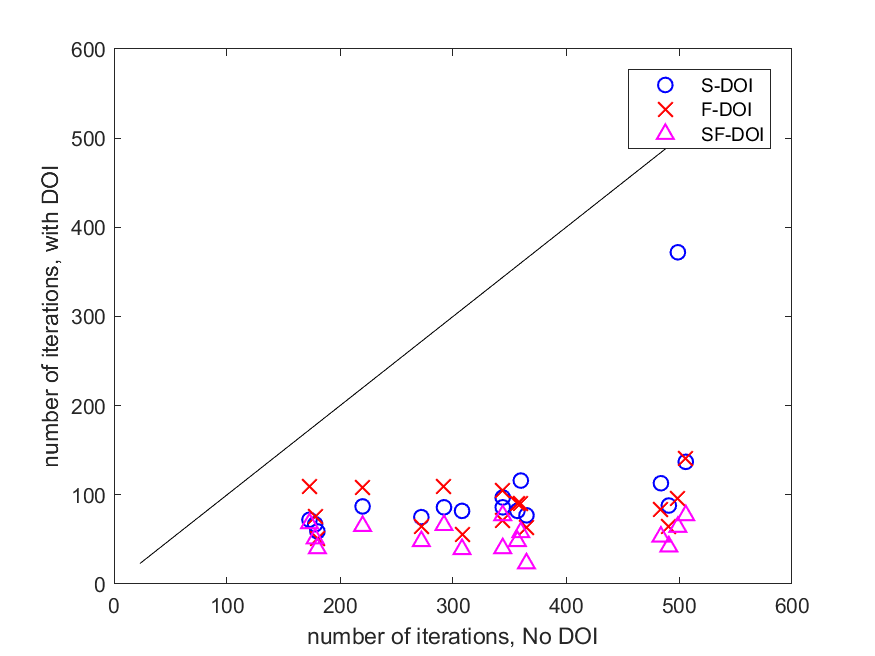}
	\caption{\citeauthor{holmberg1999exact} dataset $(|N|=200)$, Aggregated plots. Relative gaps are displayed as relative difference between upper and the maximum lower bounds.
		\textbf{(Top Left):} Average relative gap over 16 problem instances as a function of time.  
		\textbf{(Top Right):} Average relative gap over 16 problem instances as a function of iterations.
		\textbf{(Bottom Left):} Comparative run times between using DOI and using no DOI for all 16 problem instances.  
		\textbf{(Bottom Right):} Comparative iterations required between using DOI and using no DOI for all 16 problem instances.}
	\label{fig:TB2}
\end{figure}

\subsection{Results on the \citeauthor{yang2012cut} dataset}

In our second set of experiments we use the benchmark dataset presented in \citet{yang2012cut}. This dataset contains a number of large instances.  We focus here on a subset of 10 instances where $|N|=200$.  Instances 1-5 have $|I|=30$ while instance 6-10 have $|I|=60$. Customer and facility locations are randomly assigned on the unit square. Costs are set as the euclidean distance between a particular customer and facility multiplied by 10 and then rounded. The ratio of total capacity to total demand ranges from 1.8 to 3.5.

We report the same data as for the \citeauthor{holmberg1999exact} dataset. In Table \ref{table:TB3} we report the computing times required to achieve the linear relaxation lower bounds. In Figure \ref{fig:TB3} we report the DOI average performance as a function of time and iterations, just as in Figure \ref{fig:TB2}. Here, although each DOI provides an improvement in the iterations, the F-DOI and the SF-DOI fail to improve upon the convergence time and in fact manage to slow down convergence overall. The computational cost here for the F-DOI is too high when compared to the iterative benefit, leading to a speedup (or a slowdown rather) of 0.2. The S-DOI still manage to provide a positive speedup in all instances, with an average speedup of 1.7.  We note here that the F-DOI and the SF-DOI can often have their benefits negated due to their heavier computational cost. They tend to provide better benefits on problems whose solutions contain denser columns relative to the size of the problem.  If the problem does not contain dense columns, the benefit cannot manage to outweigh the computational cost, as is seen here.

\begin{table}[!hbtp]
	\centering
	\scalebox{0.8}{
	\begin{tabular}{|c|c c c c|c c c|} 
		\hline
		& \multicolumn{4}{c|}{\bf Time (sec)} & \multicolumn{3}{c|}{\bf Speedup}\\
		\bf Instance & \bf Standard & \bf S-DOI & \bf F-DOI & \bf SF-DOI & \bf S-DOI & \bf F-DOI & \bf SF-DOI \\
		\hline
		1 & 47.5 & 29.1 & 217.5 & 151.4 & 1.6 & 0.2 & 0.3 \\ 
		2 & 40.2 & 26.0 & 222.1 & 137.8 & 1.5 & 0.2 & 0.3 \\
		3 & 53.0 & 28.3 & 213.2 & 143.0 & 1.9 & 0.2 & 0.4 \\
		4 & 74.6 & 36.6 & 276.2 & 171.1 & 2.0 & 0.3 & 0.4 \\
		5 & 131.5 & 89.2 & 572.6 & 269.9 & 1.5 & 0.2 & 0.5 \\
		6 & 58.0 & 30.3 & 249.4 & 156.9 & 1.9 & 0.2 & 0.4 \\
		7 & 49.5 & 27.1 & 233.8 & 142.5 & 1.8 & 0.2 & 0.3 \\
		8 & 68.9 & 40.5 & 312.3 & 159.3 & 1.7 & 0.2 & 0.4 \\
		9 & 53.9 & 27.7 & 236.5 & 127.9 & 1.9 & 0.2 & 0.4 \\
		10 & 56.7 & 33.6 & 249.6 & 148.5 & 1.7 & 0.2 & 0.4 \\
		\hline
		\bf mean & \bf 63.4 & \bf 36.8 & \bf 278.3 & \bf 160.8 & \bf 1.8 & \bf 0.2 & \bf 0.4 \\
		\bf median & \bf 55.3 & \bf 29.7 & \bf 243.0 & \bf 149.9 & \bf 1.8 & \bf 0.2 & \bf 0.4 \\
		\hline
	\end{tabular}}
	\caption{\citeauthor{yang2012cut} dataset $(|N|=200)$ runtime results}
	\label{table:TB3}
\end{table}

\begin{figure}[!hbtp]
	\includegraphics[width=0.45\linewidth]{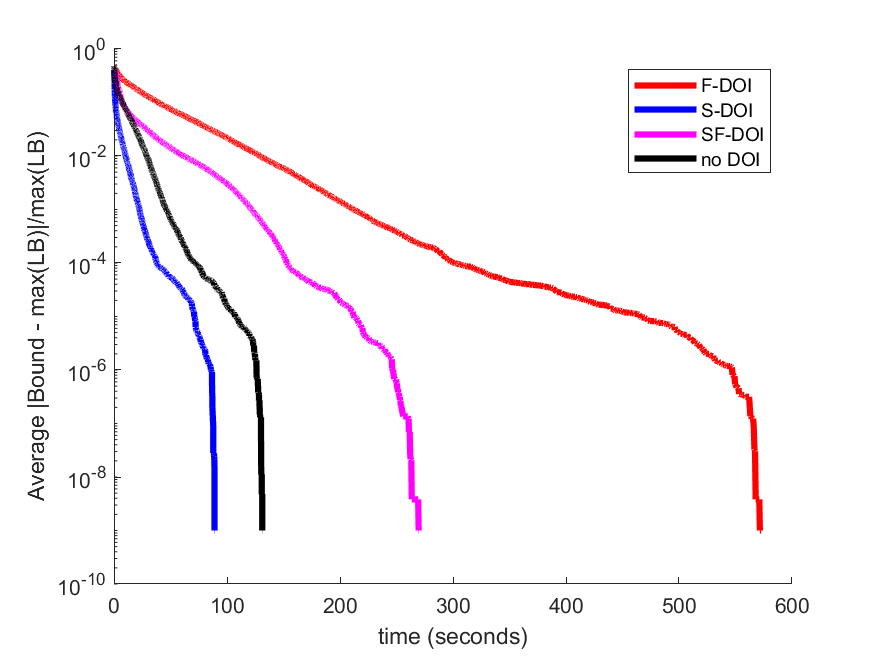}
	\includegraphics[width=0.45\linewidth]{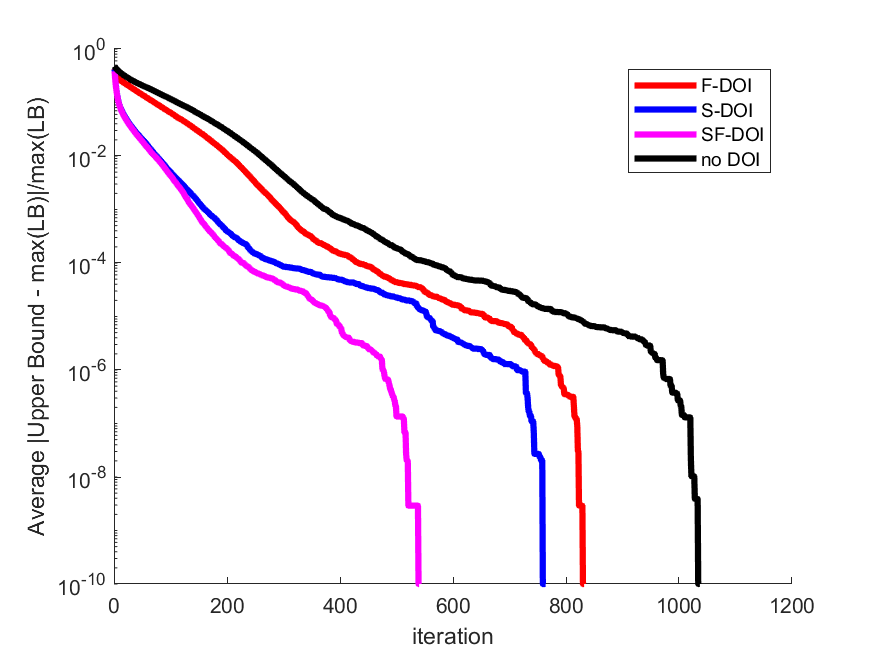}\\
	\includegraphics[width=0.45\linewidth]{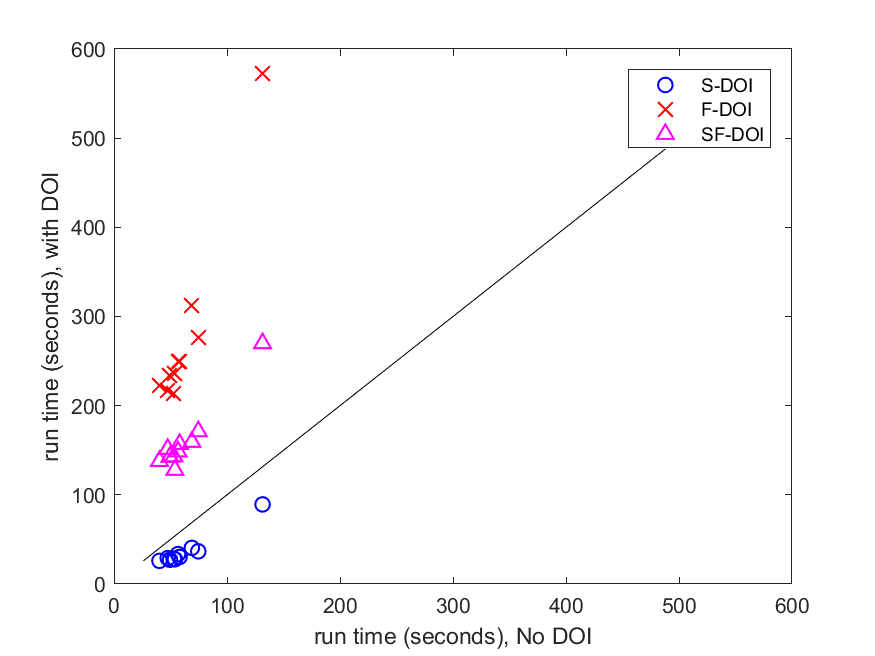}
	\includegraphics[width=0.45\linewidth]{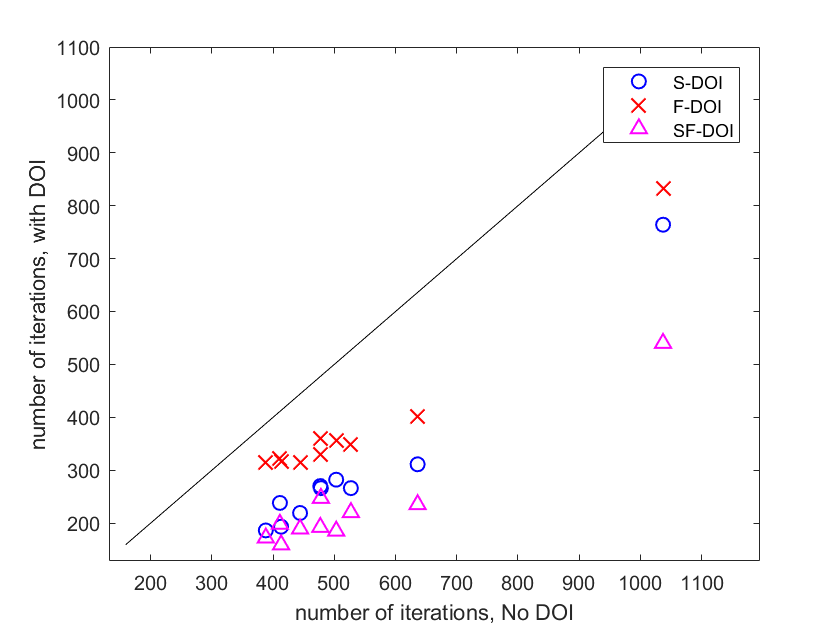}
	\caption{\citeauthor{yang2012cut} dataset $(|N|=200)$, Aggregated plots.  Relative gaps are displayed as the relative difference between upper and maximum lower bound.
		\textbf{(Top Left):} Average relative gap over 10 problem instances as a function of time.  
		\textbf{(Top Right):} Average relative gap over 10 problem instances as a function of iterations.
		\textbf{(Bottom Left):} Comparative run times between using DOI and using no DOI for all 10 problem instances.  
		\textbf{(Bottom Right):} Comparative iterations required between using DOI and using no DOI for all 10 problem instances.}
	\label{fig:TB3}
\end{figure}

\subsection{DOI and their effect on problems with dense columns}

We note that problems with denser columns in the final solution can lead to more degeneracy, allowing the DOI to provide more benefit. Facility capacities provide a hard ceiling on the size of any potential column. We aim here to relax this ceiling and see the effect on the performance of each of the DOI. Specifically, we take the same problem instances from \citet{holmberg1999exact} and \citet{yang2012cut} in the previous sections and boost each facility capacity by a factor of 2, 3, and 4. For each facility we use $K_{i}'=LK_{i}$ as the facility capacity where $K_{i}$ was the original facility capacity and $L$ is the factor increase.

The speedup results for both datasets are show in Table \ref{table:TB2TB3_upped_capacity}. Here we see a general trend of improvement in speedup as $L$ increases. The F-DOI and SF-DOI show the most marked improvement. The F-DOI goes from an average speedup of 2.0 and 0.2 at $L=1$ on \citeauthor{holmberg1999exact} and \citeauthor{yang2012cut} datasets respectively to 9.2 and 4.4 at $L=4$. The S-DOI also shows overall improvement, going from an average speedup of 3.3 to 8.6 on \citeauthor{holmberg1999exact} and from 1.8 to 4.0 on \citeauthor{yang2012cut}. The SF-DOI shows the greatest average performance at high capacity levels with a speedup of 16.1 and 11.5 at $L=4$ on \citeauthor{holmberg1999exact} and \citeauthor{yang2012cut} datasets, respectively. We note that on most problem instances where S-DOI and F-DOI both perform well, the SF-DOI can often experience significant diminishing returns in employing both DOI. On certain problems, however, this stark pattern of diminishing returns is not as prevalent and the SF-DOI can do better than both other DOI separately. We find this occurs frequently on heavily degenerate problems where standard CG takes particularly long to converge.

\begin{table}[!hbtp]
	\centering
	\scalebox{0.75}{
	\begin{tabular}{|c|c|c|c c c c|c c c|} 
		\hline
		\multicolumn{3}{|c|}{} & \multicolumn{4}{c|}{\bf Time (sec)} & \multicolumn{3}{c|}{\bf Speedup}\\
		\multicolumn{3}{|c|}{\bf Instance} & \bf Standard & \bf S-DOI & \bf F-DOI & \bf SF-DOI & \bf S-DOI & \bf F-DOI & \bf SF-DOI \\
		\hline
		\parbox[t]{2mm}{\multirow{8}{*}{\rotatebox[origin=c]{90}{Holmberg et al.}}} & \multirow{4}{*}{mean} & $L=1$ & 43.4 & 14.6 & 27.9 & 20.7 & 3.3 & 2.0 & 3.2 \\ 
		& &$L=2$ & 136.8 & 19.7 & 26.3 & 14.4 & 25.3 & 5.4 & 11.0 \\
		& &$L=3$ & 240.4 & 33.8 & 28.5 & 16.2 & 8.3 & 8.2 & 14.4 \\
		& &$L=4$ & 300.0 & 40.5 & 32.4 & 19.1 & 8.6 & 9.2 & 16.1 \\
		\cline{2-10}
		& \multirow{4}{*}{median} & $L=1$ & 44.3 & 10.9 & 21.8 & 14.9 & 3.4 & 1.8 & 2.8 \\ 
		& &$L=2$ & 117.6 & 12.8 & 25.5 & 13.1 & 16.7 & 5.8 & 9.8 \\
		& &$L=3$ & 211.2 & 24.7 & 27.4 & 15.9 & 8.4 & 8.4 & 13.6 \\
		& &$L=4$ & 255.1 & 29.8 & 29.0 & 18.9 & 9.0 & 8.6 & 14.8 \\
		\hline
		\parbox[t]{2mm}{\multirow{8}{*}{\rotatebox[origin=c]{90}{Yang et al.}}} & \multirow{4}{*}{mean} & $L=1$ & 63.4 & 36.8 & 278.3 & 160.8 & 1.8 & 0.2 & 0.4 \\ 
		& &$L=2$ & 198.5 & 80.8 & 382.2 & 197.0 & 2.4 & 0.5 & 1.1 \\
		& &$L=3$ & 581.8 & 191.3 & 326.6 & 173.9 & 3.0 & 1.9 & 4.2 \\
		& &$L=4$ & 1269.0 & 218.6 & 231.9 & 123.5 & 4.0 & 4.6 & 11.5 \\
		\cline{2-10}
		& \multirow{4}{*}{median} & $L=1$ & 55.3 & 29.7 & 243.0 & 149.9 & 1.8 & 0.2 & 0.4 \\ 
		& &$L=2$ & 163.6 & 72.4 & 326.0 & 181.3 & 2.3 & 0.5 & 0.8 \\
		& &$L=3$ & 360.3 & 134.9 & 311.0 & 154.1 & 2.7 & 1.1 & 2.2 \\
		& &$L=4$ & 405.3 & 159.2 & 210.2 & 114.9 & 2.8 & 2.3 & 4.5 \\
		\hline
	\end{tabular}}
	\caption{Runtime results for increased capacity. New capacity $K_{i}'=LK_{i}$ for each facility}
	\label{table:TB2TB3_upped_capacity}
\end{table}

\subsection{Results on newly generated random instances}

To assess our DOI further, we have decided to generate two new datasets. The first dataset has a specific construction where assignment costs are untruncated Euclidean distances between facilities and customers. We refer to these problems as the structured problems. We set $|N|=250$ and $|I|=50$ and generate 50 independent and identically distributed random instances. For each instance, customer and facility locations are randomly generated uniformly on the 2-D unit plane. The associated costs between facilities and customers is set as the euclidean distance between their respective locations. Each facility is given an instantiation cost $f_{i}=5$ and capacity $K_{i}=150$.  Each customer is assigned a random demand drawn uniformly over $\{1,2,3,4,5\}$. Table \ref{table:structured} provides aggregate runtimes and Figure \ref{fig:structured} aggregate plots of the DOI performance.

In the second class of random problem instances we abandon the structure defined previously and instead assign costs between facilities and customers devoid of any underlying structure. We refer to these problems as the unstructured problems. Assignment costs are randomly generated over a uniform distribution on the interval (0,1). We set $|N|=250$ and $|I|=50$ and generate 50 independent and identically distributed random instances. Each facility is given an instantiation cost $f_{i}=5$ and capacity $K_{i}=150$. Each customer is assigned a random demand drawn uniformly over $\{1,2,3,4,5\}$. Runtime results are shown in Table \ref{table:unstructured} and aggregate plots are shown in Figure \ref{fig:unstructured}. Both datasets can be downloaded from \url{http://claudio.contardo.org}.

\begin{table}[!hbtp]
	\centering
	\scalebox{0.87}{
	\begin{tabular}{|c|c c c c|c c c|} 
		\hline
		\multirow{2}{*}{ } & \multicolumn{4}{c|}{\bf Time (sec)} & \multicolumn{3}{c|}{\bf Speedup}\\
		& \bf Standard & \bf S-DOI & \bf F-DOI & \bf SF-DOI & \bf S-DOI & \bf F-DOI & \bf SF-DOI \\
		\hline
		mean & 3252.0 & 37.3 & 226.6 & 52.7 & 113.6 & 14.4 & 130.2 \\ 
		\hline
		median & 766.1 & 35.0 & 206.5 & 51.8 & 22.7 & 4.4 & 17.4 \\ 
		\hline
	\end{tabular}}
	\caption{Average runtime over 50 structured problem instances.}
	\label{table:structured}
\end{table}

\begin{figure}[!hbtp]
	\includegraphics[width=0.45\linewidth]{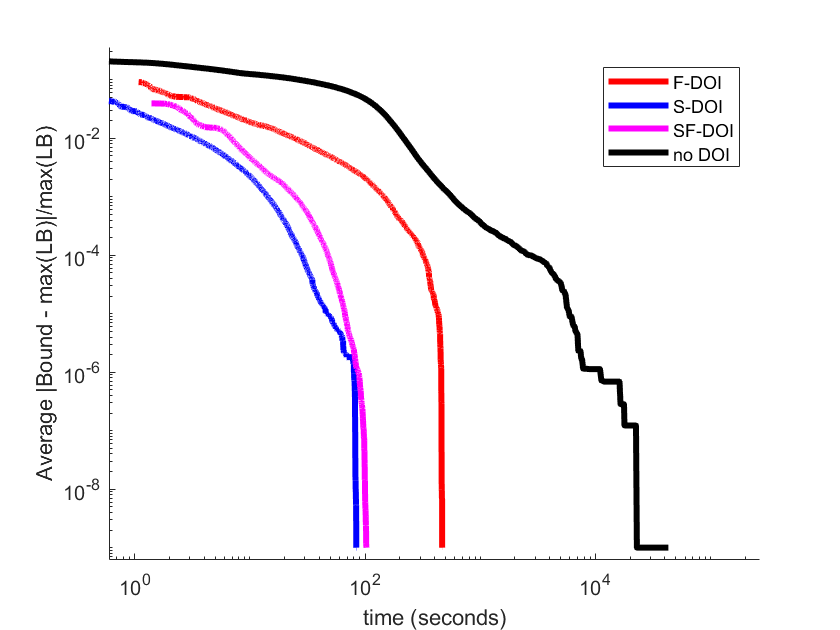}
	\includegraphics[width=0.45\linewidth]{results/structured/aggregate_time_full.png}\\
	\includegraphics[width=0.45\linewidth]{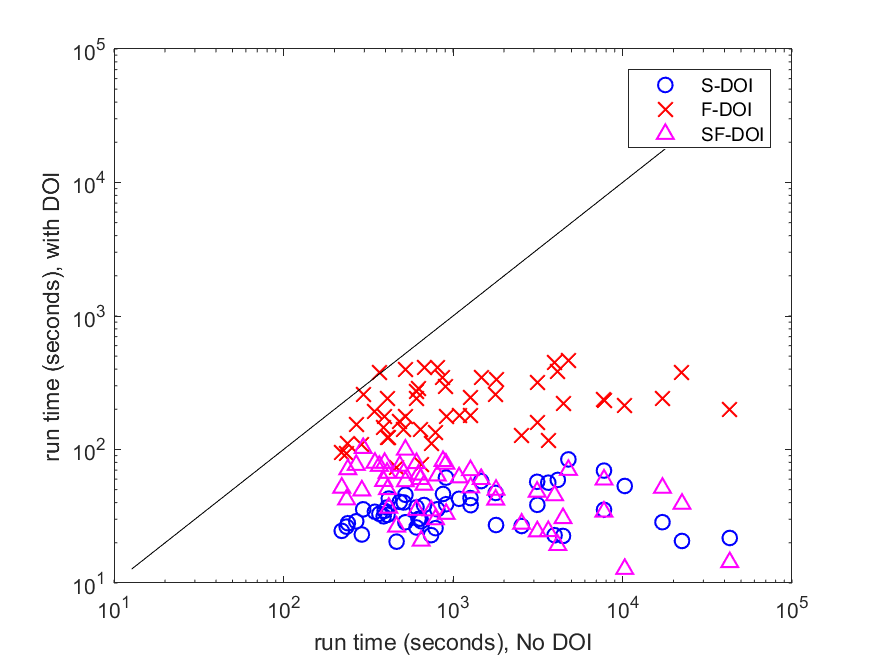}
	\includegraphics[width=0.45\linewidth]{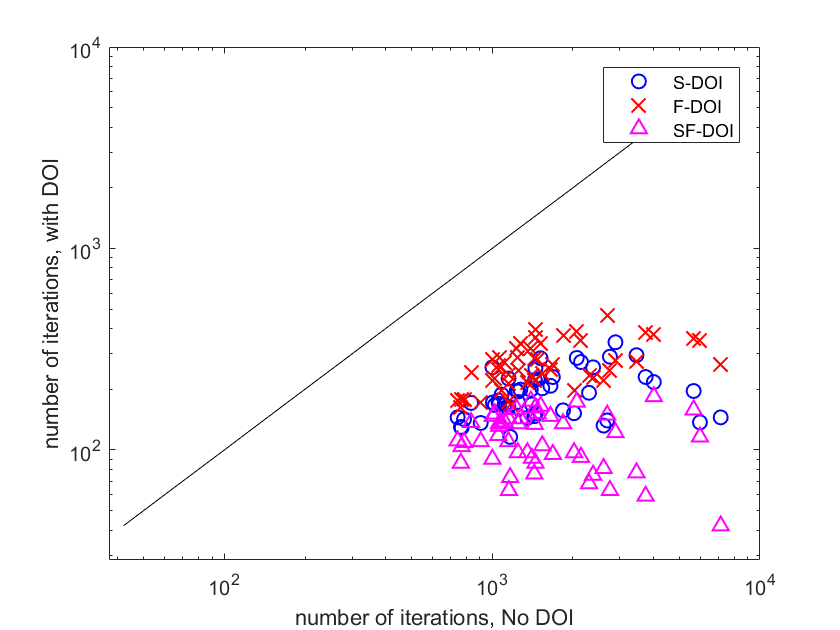}
	\caption{Structured SSCFLP aggregate plots.  Relative gaps are displayed as the relative difference between upper and maximum lower bound.
		\textbf{(Top Left):} Average relative gap over 50 problem instances as a function of time.  
		\textbf{(Top Right):} Average relative gap over 50 problem instances as a function of iterations.
		\textbf{(Bottom Left):} Comparative run times between using DOI and using no DOI for all 50 problem instances.  
		\textbf{(Bottom Right):} Comparative iterations required between using DOI and using no DOI for all 50 problem instances.}
	\label{fig:structured}
\end{figure}


\begin{table}[!hbtp]
	\centering
	\scalebox{0.87}{
	\begin{tabular}{|c|c c c c|c c c|} 
		\hline
		\multirow{2}{*}{ } & \multicolumn{4}{c|}{\bf Time (sec)} & \multicolumn{3}{c|}{\bf Speedup}\\
		& \bf Standard & \bf S-DOI & \bf F-DOI & \bf SF-DOI & \bf S-DOI & \bf F-DOI & \bf SF-DOI \\
		\hline
		mean & 89.0 & 99.5 & 52.9 & 80.5 & 0.9 & 1.7 & 1.1 \\
		\hline
		median & 88.7 & 98.7 & 53.7 & 82.4 & 0.9 & 1.7 & 1.1 \\
		\hline
	\end{tabular}}
	\caption{Average runtime over 50 unstructured problem instances.}
	\label{table:unstructured}
\end{table}

\begin{figure}[!hbtp]
	\includegraphics[width=0.45\linewidth]{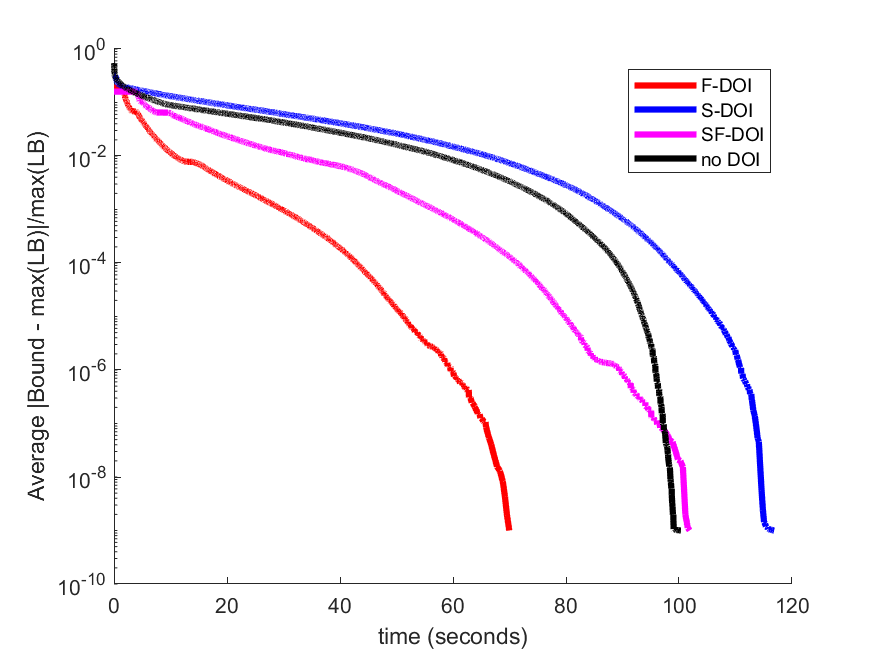}
	\includegraphics[width=0.45\linewidth]{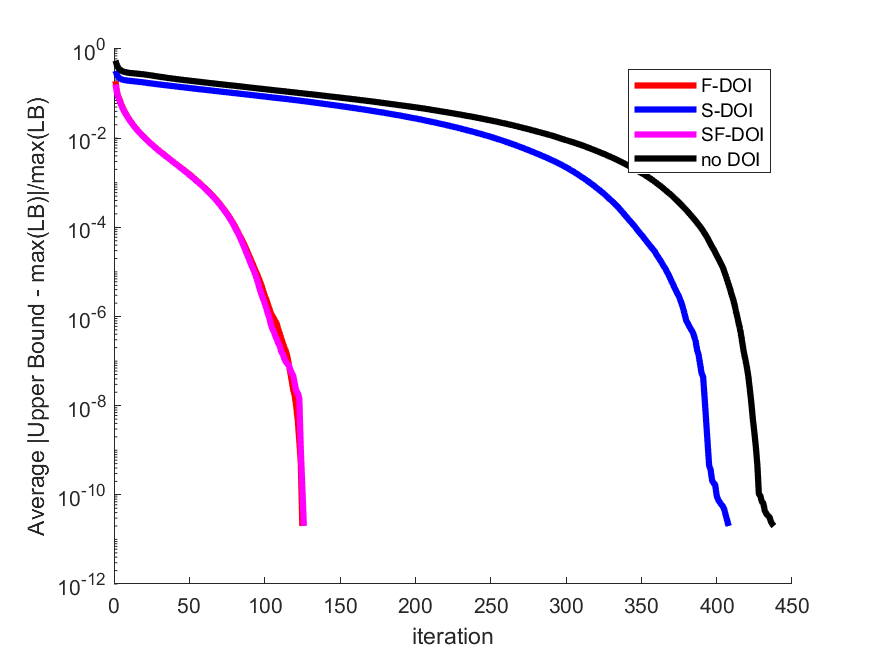}\\
	\includegraphics[width=0.45\linewidth]{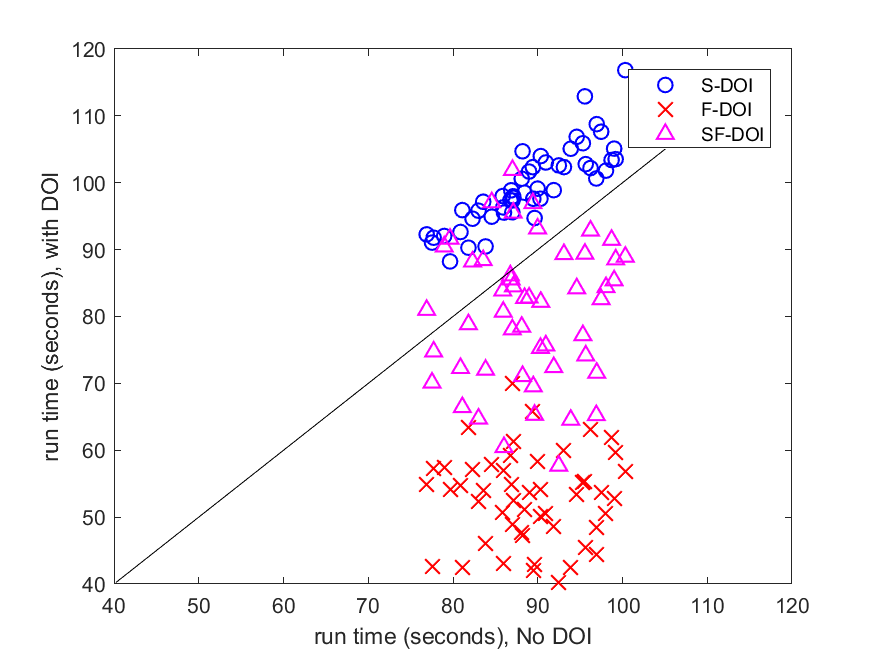}
	\includegraphics[width=0.45\linewidth]{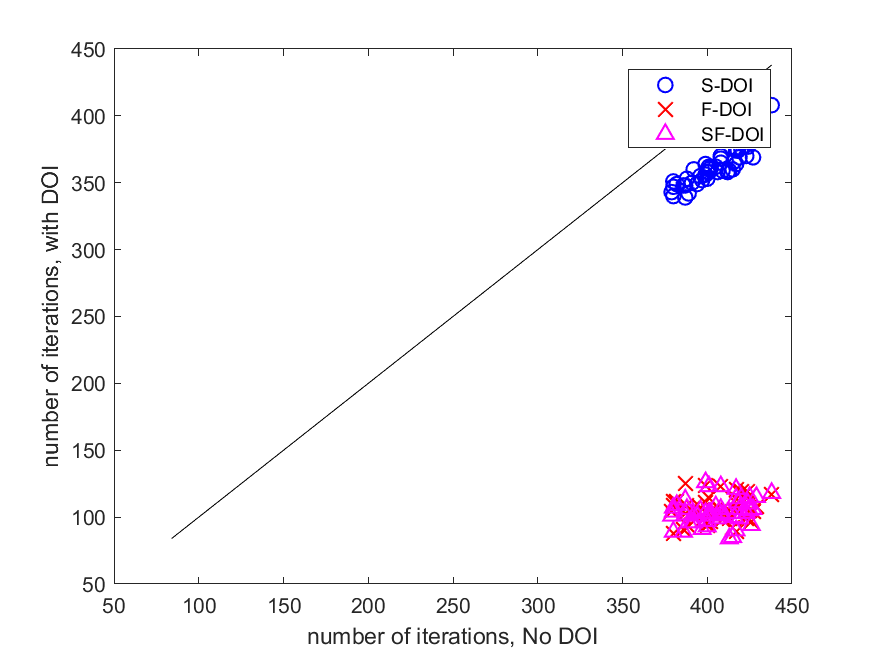}
	\caption{Unstructured SSCFLP aggregate plots.  Relative gaps are displayed as the relative difference between upper and maximum lower bound.
		\textbf{(Top Left):} Average relative gap over 50 problem instances as a function of time.  
		\textbf{(Top Right):} Average relative gap over 50 problem instances as a function of iterations.
		\textbf{(Bottom Left):} Comparative run times between using DOI and using no DOI for all 50 problem instances.  
		\textbf{(Bottom Right):} Comparative iterations required between using DOI and using no DOI for all 50 problem instances.}
	\label{fig:unstructured}
\end{figure}

We immediately see from the structured problem dataset the vast improvement achievable with the S-DOI on highly degenerate problems with an underlying structure behind the assignment costs.  The S-DOI and the SF-DOI manage median speedups of 22.7 and 17.4 respectively. The F-DOI also do well with a median speedup of 4.4. It can be observed from the relative runtime plots that the S-DOI actually outperform the SF-DOI for most problem instances. However, for particularly difficult problems for standard CG to solve, both DOI appear very competitive with one another. The SF-DOI see a massive improvement in the number of iterations required on these types of challenging problems.

From the unstructured problem dataset we observe a limitation of the S-DOI. The S-DOI fail to provide a significant benefit. This is due to the poor correlation between relative customer costs across facilities. This prevents the S-DOI from finding low $\rho_s$ values to effectively restrict the dual space. We note that the F-DOI do not share this limitation as they perform relatively well on these problems, achieving an average speedup of 1.7. Just as the S-DOI were more robust to problems with lower facility capacities, we see here that the F-DOI are more robust to problems with weaker underlying structure.

\section{Concluding remarks}
\label{section:conclusions}

In this document we introduced a new approach to accelerate the convergence of column generation when applied to set-covering-based formulations. The proposed mechanism uses little-to-no specific problem knowledge to derive two classes of dual-optimal inequalities. The Smooth-DOI (S-DOI) estimates the gain that would be obtained by replacing items from a column by items not originally covered, and applies a penalty for such swapping operation. The Flexible-DOI (F-DOI) describe rebates for removing items from columns.  These rebates are not sufficiently large to weaken the relaxation of the master problem.  
We show that our DOI can be combined without any additional implementation effort, and refer to such combination as the Smooth-Flexible-DOI (SF-DOI). We have assessed the efficiency of the new DOI to accelerate the convergence of some degenerate problem instances of the single-source capacitated facility location problem. The computational evidence shows that the new dual inequalities are efficient at reducing by a factor of up to $130\times$ the CPU times required to compute the linear relaxation dual bounds on structured problems with high a correlation between node positioning and assignment costs and loose facility capacities.
In future work we seek to apply our approach to capacitated vehicle routing and extend our approach to operate in the context of branch and price \citep{barnprice}.  The extension of our approach to consider valid inequalities such as subset-row inequalities \citep{jepsen2008subset} presents another promising avenue for exploration. 

\bibliographystyle{abbrvnat} 
\bibliography{plots/col_gen_bib}
\end{document}